\documentclass[a4paper,reqno,11pt, oneside]{amsart}

\usepackage{amssymb}
\usepackage{amstext}
\usepackage{amsmath}
\usepackage{amscd}
\usepackage{amsthm}
\usepackage{amsfonts}
\usepackage{enumerate}
\usepackage{graphicx}
\usepackage{color}
\usepackage{here} % 図や表を強制的に出力させる
\usepackage{bm} % 太字ベクトル

\usepackage{tikz}
\usepackage{mathtools}
\usepackage{latexsym}
\usepackage{booktabs}
\makeatletter
  
  \@addtoreset{equation}{section}
\makeatother

\newcommand{\calI}{\mathcal{I}}
\newcommand{\calE}{\mathcal{E}}

\newcommand{\calN}{\mathcal{N}}

\newcommand{\calM}{\mathcal{M}}

\newcommand{\ZZ}{\mathbb{Z}}

\newcommand{\RR}{\mathbb{R}}

\newcommand{\kk}{\Bbbk}

\newcommand{\eb}{\mathbf{e}}

\newcommand{\Hom}{\operatorname{Hom}}

\def\opn#1#2{\def#1{\operatorname{#2}}} % to make operators
\opn\Cl{Cl} \opn\conv{conv} \opn\deg{deg} \opn\rank{rank} \opn\Spec{Spec} 
\opn\cone{cone} \opn\End{End} \opn\Hom{Hom} \opn\mod{mod} \opn\gldim{gldim} \opn\star{star}

\newtheorem{thm}{Theorem}[section]
\newtheorem{cor}[thm]{Corollary}
\newtheorem{lem}[thm]{Lemma}
\newtheorem{prop}[thm]{Proposition}

\theoremstyle{definition}
\newtheorem{defi}[thm]{Definition}
\newtheorem{ex}[thm]{Example}

\theoremstyle{remark}
\newtheorem{rem}[thm]{Remark}

\begin{document}

\title{Levelness versus almost Gorensteinness of edge rings of complete multipartite graphs}
\author{Akihiro Higashitani}
\author{Koji Matsushita}

\address[A. Higashitani]{Department of Pure and Applied Mathematics, Graduate School of Information Science and Technology, Osaka University, Suita, Osaka 565-0871, Japan}
\email{higashitani@ist.osaka-u.ac.jp}
\address[K. Matsushita]{Department of Pure and Applied Mathematics, Graduate School of Information Science and Technology, Osaka University, Suita, Osaka 565-0871, Japan}
\email{k-matsushita@ist.osaka-u.ac.jp}


\maketitle

%%%---abstract---%%%
\begin{abstract} 
Levelness and almost Gorensteinness are well-studied properties on graded rings as a generalized notion of Gorensteinness. 
In the present paper, we study those properties for the edge rings of the complete multipartite graphs, 
denoted by $\kk[K_{r_1,\ldots,r_n}]$ with $1 \leq r_1 \leq \cdots \leq r_n$. 
We give the complete characterization of which $\kk[K_{r_1,\ldots,r_n}]$ is level in terms of $n$ and $r_1,\ldots,r_n$. 
Similarly, we also give the complete characterization of which $\kk[K_{r_1,\ldots,r_n}]$ is almost Gorenstein in terms of $n$ and $r_1,\ldots,r_n$. 
\end{abstract}

\bigskip

\section{Introduction}

\subsection{Backgrounds}

Cohen--Macaulay (local or graded) rings and Gorenstein (local or graded) rings are definitely the most important properties 
and play the crucial roles in the theory of commutative algebras. 
However, there are quite many examples which are Cohen--Macaulay but not Gorenstein. 
Thus, many researchers working on commutative algebras have been attempting the introduction of good ``intermediate'' classes of those two properties. 
Thanks to those previous studies, many classes of Cohen--Macaulay graded rings which are not Gorenstein have been defined and those theories have been developed. 
In the present paper, we concentrate on two well-studied properties, {\em level} rings and {\em almost Gorenstein} homogeneous rings. 
For the precise definitions of level rings and almost Gorenstein homogeneous rings, see Section~\ref{sec:def}. 
For example, the characterizations of both levelness and almost Gorensteinness of Hibi rings are given by Miyazaki (\cite{M17} and \cite{M18}). 
Those characterizations will be mentioned in Section~\ref{sec:poset} in detail. 

It is clear that level rings and almost Gorenstein rings are Gorenstein, 
but it is hard to determine how different between levelness and Gorensteinness as well as almost Gorensteinness and Gorensteinness. 
Even so, it is a naive problem to analyze the differences of those properties. 
Moreover, it is also natural to think of the difference of levelness and almost Gorensteinness for homogeneous rings 
and compare these properties from points of view of an extension of Gorenstein homogeneous rings. 
For those purpose, we restrict the objects of homogeneous rings. 
The central objects of the present paper are the edge rings of the complete multipartite graphs.

\subsection{Edge polytopes and edge rings}
%Throughout the present paper, all graphs are finite and have no loop and no multiple edge. 
We recall the definition of edge rings. Note that edge rings are homogeneous rings arising from graphs 
and regarded as toric rings associated to some convex polytopes, which are called edge polytopes. 
See, e.g., \cite[Section 10]{Villa} or \cite[Section 5]{HHO} for the introduction to edge rings. 

For a positive integer $d$, let $[d]:=\{1,\ldots,d\}$. 
Consider a finite simple graph $G$ on the vertex set $V(G)=[d]$ with the edge set $E(G)$. 
Given an edge $e=\{i,j\} \in E(G)$, let $\rho(e):=\eb_i+\eb_j$, where $\eb_i$ denotes the $i$-th unit vector of $\RR^d$ for $i=1,\ldots,d$. 
We define the convex polytope associated to $G$ as follows: 
$$P_G:=\conv(\{\rho(e) : e \in E(G)\}) \subset \RR^d.$$ 
We call $P_G$ the \textit{edge polytope} of $G$. We also define the {\em edge ring} of $G$, denoted by $\kk[G]$, 
as a subalgebra of the polynomial ring $\kk[{\bf t}]=\kk[t_1,\ldots,t_d]$ in $d$ variables over a field $\kk$ as follows: 
$$\kk[G] := \kk[t_it_j : \{i,j\} \in E(G)].$$ 
This is actually a monoid $\kk$-algebra associated to the monoid $\ZZ_{\geq 0}(P_G \cap \ZZ^d)$, 
i.e., the edge ring is the toric ring (a.k.a. the polytopal monomial subring) of the edge polytope. 
Note that $\dim P_G=d-1$ if $G$ is non-bipartite (\cite[Proposition 1.3]{OH98}). 
Thus, we conclude that the Krull dimension of $\kk[G]$, denoted by $\dim\kk[G]$, is equial to $d$ if $G$ is non-bipartite.

\subsection{Edge rings of complete multipartite graphs}

We also recall what complete multipartite graphs are. 
Let $K_{r_1,\ldots,r_n}$ be the graph on the vertex set $\bigsqcup_{k=1}^n V_k$, $|V_k|=r_k$ for $k=1,\ldots,n$ and $1\le r_1\le \cdots \le r_n$, 
with the edge set $\{\{u,v\} : u \in V_i, v \in V_j, 1 \leq i < j \leq n\}.$ 
This graph $K_{r_1,\ldots,r_n}$ is called the {\em complete multipartite graph} with type $(r_1,\ldots,r_n)$. 
We always denote the number of vertices of $K_{r_1,\ldots,r_n}$ by $d$, i.e., $d=\sum_{i=1}^nr_i$. 

The edge polytope and the edge ring of $K_{r_1,\ldots,r_n}$ were investigated in \cite[Section 2]{OH00}. 
For example, the Ehrhart polynomial of $P_{K_{r_1,\ldots,r_n}}$ (i.e., the Hilbert function of $\kk[K_{r_1,\ldots,r_n}]$) 
is completely determined in \cite[Theorem 2.6]{OH00}. 
Moreover, the Gorensteinness of $\kk[K_{r_1,\ldots,r_n}]$ is proved as follows: 
\begin{prop}[Characterization of Gorensteinness, {\cite[Remark 2.8]{OH00}}]
Let $1 \leq r_1 \leq \cdots \leq r_n$ and let $d=\sum_{i=1}^n r_i$, where $n \geq 2$. 
Then the edge ring of the complete multipartite graph $K_{r_1,\ldots,r_n}$ is Gorenstein if and only if 
\begin{itemize}
\item $n=2$ and $(r_1,r_2) \in \{(1,m), (m,m) : m \geq 1\}$; 
\item $n=3$ and $1 \leq r_1 \leq r_2 \leq r_3 \leq 2$; 
\item $n=4$ and $r_1=\cdots=r_4=1$. 
\end{itemize}
\end{prop}
This proposition is a direct consequence of \cite{DH}. 

In \cite{HM}, the authors of the present paper investigated $\kk[K_{r_1,\ldots,r_n}]$ from different points of view. 
For example, it is proved that the class group of $\kk[K_{r_1,\ldots,r_n}]$ is isomorphic to $\ZZ^n$ if $n =3$ with $r_1 \geq 2$ or $n \geq 4$. 
Moreover, the authors also discuss its conic divisorial ideals and construct non-commutative crepant resolutions 
for Gorenstein edge rings of $K_{r_1,\ldots,r_n}$.

\subsection{Main Results}

The goal of the present paper is to determine when $\kk[K_{r_1,\ldots,r_n}]$ is level or almost Gorenstein, where $1 \leq r_1 \leq \cdots \leq r_n$. 
In the case where $n=2$ or $n=3$ with $r_1=1$, Proposition~\ref{prop:order} says that the edge ring $\kk[K_{r_1,\ldots,r_n}]$ is isomorphic to a certain Hibi ring. 
Thus, the characterizations can be obtained from the results on Hibi rings. See Section~\ref{sec:poset}. 
Hence, our main concern is in the case where $n=3$ with $r_1 \geq 2$ or $n \geq 4$.

The first main result is the characterization of the levelness of $\kk[K_{r_1,\ldots,r_n}]$: 
\begin{thm}[Characterization of levelness]\label{thm:level}
Let $1 \leq r_1 \leq \cdots \leq r_n$ and let $d=\sum_{i=1}^n r_i$, where $n \geq 2$. 
Then the edge ring of the complete multipartite graph $K_{r_1,\ldots,r_n}$ is level if and only if $n$ and $(r_1,\ldots,r_n)$ satisfy one of the following: 
\begin{enumerate}
\item[{\em (i)}] $n=2$; 
\item[{\em (ii)}] $n=3$ and $(r_1,r_2,r_3) \in \{(1,1,m) : m \geq 1\} \cup \{(1,2,m) : m \geq 2\}$; 
\item[{\em (iii)}] $n=3$ and $(r_1,r_2,r_3) \in \{(2,2,m) : m \geq 2\} \cup \{(3,3,3)\}$; 
\item[{\em (iv)}] $n=4$ and $(r_1,r_2,r_3,r_4) \in \{(1,1,1,m) : m \geq 1\}$; 
\item[{\em (v)}] $n=5$ and $r_1=\cdots=r_5=1$. 
\end{enumerate}
\end{thm}
Note that the first two cases come from the results on Hibi rings. See Proposition~\ref{prop:level}. %as mentioned in Proposition~\ref{prop:order}, 
%so those cases are proved in a different way. See Section~\ref{sec:poset}. 

The second main result is the characterization of the almost Gorensteinness of $\kk[K_{r_1,\ldots,r_n}]$: 
\begin{thm}[Characterization of almost Gorensteinness]\label{thm:almGor}
Let $1 \leq r_1 \leq \cdots \leq r_n$ and let $d=\sum_{i=1}^n r_i$, where $n \geq 2$. 
Then the edge ring of the complete multipartite graph $K_{r_1,\ldots,r_n}$ is almost Gorenstein 
if and only if $n$ and $(r_1,\ldots,r_n)$ satisfy one of the following: 
\begin{enumerate}
\item[{\em (i)}] $n=2$ and $(r_1,r_2) \in \{(1,m),(m,m) : m \geq 1\} \cup \{(2,m) : m \geq 2\}$; 
\item[{\em (ii)}] $n=3$ and $(r_1,r_2,r_3) \in \{(1,1,m), (1,m,m) : m \geq 1\}$; 
\item[{\em (iii)}] $n=3$ and $(r_1,r_2,r_3)=(2,2,2)$; 
\item[{\em (iv)}] $n=4$ and $(r_1,r_2,r_3,r_4) \in \{(1,1,m,m) : m \geq 1\}$; 
\item[{\em (v)}] $n \geq 4$ and $(r_1,\ldots,r_{n-1},r_n)=(1,\ldots,1,n-3)$. 
\item[{\em (vi)}] $n$ is even with $n \geq 6$ and $r_1=\cdots=r_n=1$; 
\end{enumerate}
\end{thm}
Note that the first two cases come from the result on Hibi rings. See Section~\ref{sec:poset} (below Proposition~\ref{prop:level}).

As an immediate corollary of those theorems, we obtain the following: 
\begin{cor}
The edge ring of the complete multipartite graph $K_{r_1,\ldots,r_n}$ is level and almost Gorenstein but not Gorenstein if and only if one of the following holds: 
\begin{itemize}
\item $n=2$ and $(r_1,r_2) \in \{(2,m) : m \geq 3\}$; 
\item $n=3$ and $(r_1,r_2,r_3) \in \{(1,1,m) : m \geq 3\}$. 
\end{itemize}
In particular, in both cases, the edge rings are isomorphic to certain Hibi rings. 
\end{cor}

\begin{ex}[$n=2$ or $n=3$]
For $K_{r_1,r_2}$, we see that $\kk[K_{r_1,r_2}]$ is always level. Moreover, $\kk[K_{1,m}]$ and $\kk[K_{m,m}]$ are Gorenstein, 
while $\kk[K_{2,m}]$ is not Gorenstein but almost Gorenstein if $m \geq 3$. 

For $K_{r_1,r_2,r_3}$, we see that $\kk[K_{r_1,r_2,r_3}]$ is 
\begin{itemize}
\item level but not Gorenstein for $K_{1,1,m}, K_{1,2,m},K_{2,2,m}$ with $m \geq 3$ and $K_{3,3,3}$; 
\item almost Gorenstein but not Gorenstein if and only if $K_{1,1,m},K_{1,m,m}$ with $m \geq 3$; 
\item level and almost Gorenstein but not Gorenstein if and only if $K_{1,1,m}$ with $m \geq 3$. 
\end{itemize}
%there is no triple $(r_1,r_2,r_3)$ such that $\kk[K_{r_1,r_2,r_3}]$ is not Gorenstein but almost Gorenstein. 
%Moreover, $\kk[K_{r_1,r_2,r_3}]$ is level if $1 \leq r_1 \leq r_2 \leq 2$ or in the case $K_{3,3,3}$. 
%while it is Gorenstein in the cases $K_{1,1,1}$, $K_{1,1,2}$, $K_{1,2,2}$ and $K_{2,2,2}$. 
\end{ex}
\begin{ex}[$n=4$]
For $K_{r_1,r_2,r_3,r_4}$, we see that $\kk[K_{r_1,r_2,r_3,r_4}]$ is 
\begin{itemize}
\item level for $K_{1,1,1,m}$; 
\item almost Gorenstein for $K_{1,1,m,m}$; 
\item Gorenstein for $K_{1,1,1,1}$. 
\end{itemize}
\end{ex}
\begin{ex}[$n \geq 5$]
In the case $n \geq 5$, $\kk[K_{r_1,\ldots,r_n}]$ is never Gorenstein, and it is level only for $K_5$. 
On the other hand, it is almost Gorenstein for $K_{2m}$ with $m \geq 3$ and $K_{\underbrace{1,\ldots,1}_{n-1},n-3}$. 
\end{ex}

\subsection{Organization}
First, in Section~\ref{sec:def}, we recall the precise definitions of level rings and almost Gorenstein homogeneous rings 
and some known results on them, which we will use in the proofs of our main results. 
Next, in Section~\ref{sec:poset}, we discuss the main theorems in the case where $n=2$ or $n=3$ with $r_1=1$, 
which can be deduced into the results on Hibi rings. Thus, we also recall the results on Hibi rings. 
Finally, in Section~\ref{sec:proof}, we give proofs of Theorems~\ref{thm:level} and \ref{thm:almGor}. 

%%%%%%%%%%%%%%%%%%%%%%%%%%%%%%%%%%%%%%%%%%%%%%%%%%%%%%%%%%%%%%%%%%%%%%%%%%
\subsection*{Acknowledgement} 
The first named author is partially supported by JSPS Grant-in-Aid for Scientific Research (C) 20K03513. 
%%%%%%%%%%%%%%%%%%%%%%%%%%%%%%%%%%%%%%%%%%%%%%%%%%%%%%%%%%%%%%%%%%%%%%%%%%

\bigskip

%%%%%%%%%%%%%%%%%%%%%%%%%%%%%%%%%%%%%%%%%%%%%%%%%%%%%%%%%%%%%%%%%%%%%%%%%%%%%%%%%%%%%%%%%%%%%%%%%%%%%%%%%%%%%%%%%%%%%%%%%%%%%
%%%%%%%%%%%%%%%%%%%%%%%%%%%%%%%%%%%%%%%%%%%%%%%%%%%%%%%%%%%%%%%%%%%%%%%%%%%%%%%%%%%%%%%%%%%%%%%%%%%%%%%%%%%%%%%%%%%%%%%%%%%%%
%%%%%%%%%%%%%%%%%%%%%%%%%%%%%%%%%%%%%%%%%%%%%%%%%%%%%%%%%%%%%%%%%%%%%%%%%%%%%%%%%%%%%%%%%%%%%%%%%%%%%%%%%%%%%%%%%%%%%%%%%%%%%
%%%%%%%%%%%%%%%%%%%%%%%%%%%%%%%%%%%%%%%%%%%%%%%%%%%%%%%%%%%%%%%%%%%%%%%%%%%%%%%%%%%%%%%%%%%%%%%%%%%%%%%%%%%%%%%%%%%%%%%%%%%%%

\section{Levelness and almost Gorensteinness of homogeneous rings}\label{sec:def}

Throughout this section, let $R$ be a Cohen--Macaulay homogeneous ring of dimension $d$ over an algebraically closed field $\kk$ with characteristic $0$. 
We recall the definitions of levelness and almost Gorensteinness and some properties on those algebras. 

Before defining them, we recall some fundamental materials. (Consult, e.g., \cite{BH} for the introduction to homogeneous rings.) 
\begin{itemize}
\item Let $\omega_R$ denote a canonical module of $R$. Let $a(R)$ denote the $a$-invariant of $R$, i.e., $a(R)=-\min\{j:(\omega_R)_j \neq 0\}$. 
\item For a graded $R$-module $M$, we use the following notation: 
\begin{itemize}
\item Let $\mu_j(M)$ denote the number of elements of minimal generators of $M$ with degree $j$ as an $R$-module, 
and let $\mu(M)=\sum_{j \in \ZZ}\mu_j(M)$, i.e., the number of minimal generators. 
\item Let $e(M)$ denote the multiplicity of $M$. Then the inequality $\mu(M) \leq e(M)$ always holds. 
\item Let $M(-\ell)$ denote the $R$-module whose grading is given by $M(-\ell)_n=M_{n-\ell}$ for any $n \in \ZZ$. 
\end{itemize}
\item Let $r(R)$ denote the Cohen--Macaulay type of $R$. Note that $r(R)=\mu(\omega_R)$. 
We see that $R$ is Gorenstein if and only if $r(R)=1$. 
%\item Let $H(M,m)$ denote the Hilbert function of $M$, i.e., $H(M,m)=\dim_\kk M_m$, and let $[[M]]$ denote the Hilbert Series of $M$, i.e., 
%$$[[M]]=\sum_{n \in \ZZ}\dim_\kk M_n t^n,$$ where $\dim_\kk$ stands for the dimension as a $\kk$-vector space. 
%Note that $H(M,m)$ can be described by a polynomial in $m$ of degree $d-1$ and its leading coefficient coincides with $e(M)/d!$. (See \cite[Section 4]{BH}.) 
%\item Write $(h_0,\ldots,h_s)$ for the $h$-vector of $R$, which is the coefficients of the polynomial appearing in the numerator of $[[M]]$, i.e., 
%$$[[R]]=\frac{\sum_{i=0}^sh_it^i}{(1-t)^d},$$ 
%where $h_s \neq 0$. We call the index $s$ the {\em socle degree} of $R$. 
%Note that $h_s \leq r(R)$ holds. Moreover, we see that $d+a(R)=s$. (See \cite[Section 4.4]{BH}.) 
\item Let $H(M,m)$ denote the Hilbert function of $M$, i.e., $H(M,m)=\dim_\kk M_m$, where $\dim_\kk$ stands for the dimension as a $\kk$-vector space. 
Note that $H(M,m)$ can be described by a polynomial in $m$ of degree $d-1$ and its leading coefficient coincides with $e(M)/d!$. (See \cite[Section 4]{BH}.) 
\item Write $(h_0,\ldots,h_s)$ for the $h$-vector of $R$, 
which is the coefficients of the polynomial appearing in the numerator of the Hilbert series of $R$, i.e., 
$$\sum_{n =0}^\infty \dim_\kk R_n t^n=\frac{\sum_{i=0}^sh_it^i}{(1-t)^d},$$ 
where $h_s \neq 0$. We call the index $s$ the {\em socle degree} of $R$. 
Note that $h_s \leq r(R)$ holds. Moreover, we see that $d+a(R)=s$. (See \cite[Section 4.4]{BH}.) 
\end{itemize}

\begin{defi}[Level, {\cite{S77}}]
We say that $R$ is \textit{level} if all the degrees of the minimal generators of $\omega_R$ are the same. 
In particular, what $R$ is Gorenstein implies what $R$ is level since $\omega_R$ is generated by a unique element if $R$ is Gorenstein. 
\end{defi}
\begin{rem}\label{rem:level_Gor}
Let $(h_0,h_1,\ldots,h_s)$ be the $h$-vector of $R$. Assume that $R$ is level. In this case, if $h_s=1$, then $R$ is Gorenstein. 
In fact, since what $R$ is level is equivalent to $h_s=r(R)$ (see, e.g., \cite[Section 5]{BH}), we obtain that $r(R)=1$. 
\end{rem}
Regarding the levelness of homogeneous domains, we know the following: 
\begin{thm}[{\cite[Corollary 3.11]{Y}}]\label{thm:Y}
Let $R$ be an almost Gorenstein homogeneous domain with its socle degree $s$. If $s=2$, then $R$ is level. 
\end{thm}

\begin{defi}[Almost Gorenstein, {\cite[Definition 1.5]{GTT}}]
We call $R$ {\em almost Gorenstein} %if there is an injection $\psi$ from $R$ to $\omega_R(-a)$ such that its cokernel $C$ satisfies $\mu(C)=e(C)$, namely, 
if there exists an exact sequence of graded $R$-modules 
\begin{align}\label{eq:exact}
0 \to R \to \omega_R(-a) \to C \to 0
\end{align}
with $\mu(C)=e(C)$. 
\end{defi}
Note that there always exists a degree-preserving injection from $R$ to $\omega_R(-a)$ if $R$ is a domain (\cite[Proposition 2.2]{H}). 
Moreover, we see that $\mu(C)=r(R)-1$ (\cite[Proposition 2.3]{H}), which we will use later. 

Regarding the almost Gorensteinness of homogeneous domains, we know the following: 
\begin{thm}[{\cite[Theorem 4.7]{H}}]\label{thm:h_s=1}
Let $R$ be an almost Gorenstein homogeneous domain and $(h_0,h_1,\ldots,h_s)$ its $h$-vector with $s \geq 2$. Then $h_s=1$. 
\end{thm}

\bigskip

%%%%%%%%%%%%%%%%%%%%%%%%%%%%%%%%%%%%%%%%%%%%%%%%%%%%%%%%%%%%%%%%%%%%%%%%%%%%%%%%%%%%%%%%%%%%%%%%%%%%%%%%%%%%%%%%%%%%%%%%%%%%%
%%%%%%%%%%%%%%%%%%%%%%%%%%%%%%%%%%%%%%%%%%%%%%%%%%%%%%%%%%%%%%%%%%%%%%%%%%%%%%%%%%%%%%%%%%%%%%%%%%%%%%%%%%%%%%%%%%%%%%%%%%%%%
%%%%%%%%%%%%%%%%%%%%%%%%%%%%%%%%%%%%%%%%%%%%%%%%%%%%%%%%%%%%%%%%%%%%%%%%%%%%%%%%%%%%%%%%%%%%%%%%%%%%%%%%%%%%%%%%%%%%%%%%%%%%%
%%%%%%%%%%%%%%%%%%%%%%%%%%%%%%%%%%%%%%%%%%%%%%%%%%%%%%%%%%%%%%%%%%%%%%%%%%%%%%%%%%%%%%%%%%%%%%%%%%%%%%%%%%%%%%%%%%%%%%%%%%%%%
%%%%%%%%%%%%%%%%%%%%%%%%%%%%%%%%%%%%%%%%%%%%%%%%%%%%%%%%%%%%%%%%%%%%%%%%%%%%%%%%%%%%%%%%%%%%%%%%%%%%%%%%%%%%%%%%%%%%%%%%%%%%%
%%%%%%%%%%%%%%%%%%%%%%%%%%%%%%%%%%%%%%%%%%%%%%%%%%%%%%%%%%%%%%%%%%%%%%%%%%%%%%%%%%%%%%%%%%%%%%%%%%%%%%%%%%%%%%%%%%%%%%%%%%%%%
%%%%%%%%%%%%%%%%%%%%%%%%%%%%%%%%%%%%%%%%%%%%%%%%%%%%%%%%%%%%%%%%%%%%%%%%%%%%%%%%%%%%%%%%%%%%%%%%%%%%%%%%%%%%%%%%%%%%%%%%%%%%%
%%%%%%%%%%%%%%%%%%%%%%%%%%%%%%%%%%%%%%%%%%%%%%%%%%%%%%%%%%%%%%%%%%%%%%%%%%%%%%%%%%%%%%%%%%%%%%%%%%%%%%%%%%%%%%%%%%%%%%%%%%%%%

\section{In the case of Hibi rings}\label{sec:poset}
In this section, we recall Hibi rings and the characterization results on Gorensteinness, levelness and almost Gorensteinness of Hibi rings.

Let $\Pi=\{p_1,\ldots,p_{d-1}\}$ be a finite partially ordered set (poset, for short) equipped with a partial order $\prec$. 
For $p,q \in \Pi$, we say that \textit{$p$ covers $q$} if $q \prec p$ and there is no $p' \in \Pi \setminus \{p,q\}$ with $q \prec p' \prec p$. 
For a subset $I \subset \Pi$, we say that $I$ is a \textit{poset ideal} of $\Pi$ if $p \in I$ and $q \prec p$ then $q \in I$. 
Let $\calI(\Pi)$ denote the set of all poset ideals in $\Pi$. Note that $\emptyset \in \calI(\Pi)$.

We define the $\kk$-algebra $\kk[\Pi]$ as a subalgebra of the polynomial ring $\kk[x_1,\ldots,x_{d-1},t]$ by setting 
$$\kk[\Pi]:=\kk[{\bf x}^I t : I \in \calI(\Pi)],$$
where ${\bf x}^I=\prod_{p_i \in I}x_i$ for $I \subset \Pi$ and each ${\bf x}^I t$ is defined to be of degree $1$. 
The standard graded $\kk$-algebra $\kk[\Pi]$ is called the \textit{Hibi ring} of $\Pi$. 
The following fundamental properties on Hibi rings were originally proved in \cite{H87}: 
\begin{itemize}
\item The Krull dimension of $\kk[\Pi]$ is equal to $|\Pi|+1$; 
\item $\kk[\Pi]$ is a Cohen--Macaulay normal domain; 
\item $\kk[\Pi]$ is an algebra with straightening laws on $\Pi$. 
\end{itemize}

We say that a poset $\Pi$ is {\em pure} if each maximal chain contained in the poset has the same length, 
where the length of the chain $p_{i_0} \prec p_{i_1} \prec \cdots \prec p_{i_\ell}$ is defined to be $\ell$. 
The following is well known: 
\begin{thm}[{\cite{H87}}]
The Hibi ring $\kk[\Pi]$ of $\Pi$ is Gorenstein if and only if $\Pi$ is pure. 
\end{thm}

%Let \begin{align*}
%\calO(\Pi)=\{(x_1,\ldots,x_{d-1}) \in \RR^{d-1} : \; x_i \geq x_j \text{ if } p_i \prec p_j \text{ in }\Pi, \;\; 
%0 \leq x_i \leq 1 \text{ for }i=1,\ldots,d-1\}.
%\end{align*}
%A convex polytope $\calO(\Pi)$ is called the \textit{order polytope} of $\Pi$. 
%It is known (\cite{S86}) that $\calO(\Pi)$ is a $(0,1)$-polytope and the vertices of $\calO(\Pi)$ one-to-one correspond to the poset ideals of $\Pi$. 
%In fact, a $(0,1)$-vector $(a_1,\ldots,a_{d-1})$ is a vertex of $\calO(\Pi)$ if and only if $\{ p_i \in \Pi : a_i =1 \}$ is a poset ideal. 

\medskip

Given positive integers $m$ and $n$, let $\Pi_{m,n}=\{p_1,\ldots,p_m,p_{m+1},\ldots,p_{m+n}\}$ be the poset equipped with the partial orders 
$p_1 \prec \cdots \prec p_m$ and $p_{m+1} \prec \cdots \prec p_{m+n}$. 
Moreover, let $\Pi_{m,n}'$ be the poset having an additional relation $p_1 \prec p_{m+n}$. 
Note that %$\Pi_{m,n}$ is the poset appearing in \cite[Example 2.6]{HN} with $t=2$, $r_1=m$ and $r_2=n$ and 
the Hibi ring $\kk[\Pi_{m,n}]$ is isomorphic to the Segre product of 
the polynomial ring with $(m+1)$ variables and the polynomial ring with $(n+1)$ variables (see, e.g., \cite[Example 2.6]{HN}). 

%We notice that both $\Pi_{m,n}$ and $\Pi_{m,n}'$ do not contain the X-shape subposet, 
%so $\calO(\Pi_{m,n})$ (resp. $\calO(\Pi_{m,n}')$) is unimodularly equivalent to $\calC(\Pi_{m,n})$ (resp. $\calC(\Pi_{m,n}')$) by Theorem~\ref{X}. 

%\begin{prop}[{\cite[Proposition 2.2]{HM}}]\label{prop:order} Let $m,n$ be positive integers. 
%\begin{itemize}
%\item[(1)] The edge polytope $P_{K_{m+1,n+1}}$ is unimodularly equivalent to the order polytope $\calO(\Pi_{m,n})$. 
%\item[(2)] The edge polytope $P_{K_{1,m,n}}$ is unimodularly equivalent to the order polytope $\calO(\Pi_{m,n}')$. 
%\end{itemize}
%In particular, the edge ring $\kk[K_{m+1,n+1}]$ (resp. $\kk[K_{1,m,n}]$) is isomorphic to the Hibi ring $\kk[\Pi_{m,n}]$ (resp. $\kk[\Pi_{m,n}']$). 
%\end{prop}
Actually, the certain edge rings are isomorphic to the Hibi rings $\kk[\Pi_{m,n}]$ and $\kk[\Pi_{m,n}']$ as follows: 
\begin{prop}[{\cite[Proposition 2.2]{HM}}]\label{prop:order} %Let $m,n$ be positive integers. 
The edge ring $\kk[K_{m+1,n+1}]$ (resp. $\kk[K_{1,m,n}]$) is isomorphic to the Hibi ring $\kk[\Pi_{m,n}]$ (resp. $\kk[\Pi_{m,n}']$). 
\end{prop}
Therefore, in the case where $n=2$ or $n=3$ with $r_1=1$, the characterization of levelness and almost Gorensteinness of 
$\kk[K_{r_1,\ldots,r_n}]$ can be deduced into those of the Hibi rings $\kk[\Pi_{m,n}]$ and $\kk[\Pi_{m,n}']$. 

Hence, in what follows, we give the characterizations of those properties for $\kk[\Pi_{m,n}]$ and $\kk[\Pi_{m,n}']$, 
which prove Theorems~\ref{thm:level} and \ref{thm:almGor} in the case where $n=2$ or $n=3$ with $r_1=1$. 

\medskip

Miyazaki gave the characterizations of levelness \cite{M17} and almost Gorensteinness \cite{M18} of Hibi rings. 
For explaining those results, we introduce some notions. 

Given a poset $\Pi$, let $\widehat{\Pi}:=\Pi \cup \{\hat{0},\hat{1}\}$, where $\hat{0}$ (resp. $\hat{1}$) denotes a new unique minimal (maximal) element. 
\begin{itemize}
\item For $x,\ y\in \Pi$ with $x\preceq y$, we set $[x,y]_\Pi:=\{z\in \Pi:x\preceq z\preceq y\}$. 
\item We define $\rank P$ to be the maximal lenth of the chains in $\Pi$, and define $\rank [x,y]_\Pi$ analogously.
\item Let $y_1,x_1,y_2,x_2,\ldots,y_t,x_t$ be a (possibly empty) sequence of elements in $\widehat{\Pi}$. 
We say that the sequence $y_1,x_1,y_2,x_2,\ldots,y_t,x_t$ satisfies \textit{condition N} if \\
(1) $x_1\ne \hat{0}$, \\
(2) $y_1\succ x_1\prec y_2 \succ x_2 \prec \cdots \prec y_t \succ x_t$, and \\
(3) $y_i \nsucceq x_j$ for any $i,j$ with $1\le i<j\le t$.
\item Let 
$$r(y_1,x_1,\ldots,y_t,x_t):=\sum_{i\in [t]}(\rank [x_{i-1},y_i]_{\widehat{\Pi}} -\rank [x_i,y_i]_{\widehat{\Pi}})+\rank [x_t,\hat{1}]_{\widehat{\Pi}},$$
where we set an empty sum to be 0 and $x_0=\hat{0}$. 
\item Given $x \in \Pi$, let 
$$\star_\Pi(x)=\{y \in \Pi : y \preceq x \text{ or }x \preceq y\}.$$
\end{itemize}

\begin{thm}[{\cite[Theorem 3.9]{M17}}]\label{levelthm}
The Hibi ring of $\Pi$ is level if and only if 
$$r(y_1,x_1,\ldots,y_t,x_t)\le \rank \widehat{\Pi}$$ holds for any sequence of elements in $\widehat{P}$ with condition N.
\end{thm}
\begin{cor}[{\cite[Corollary 3.10]{M17}}]\label{levelcor}
If $[x,\hat{1}]_{\widehat{\Pi}}$ is pure for any $x\in \Pi$, then $\kk[\Pi]$ is level. 
\end{cor}

By using those results by Miyazaki, we can prove the following: 
\begin{prop}\label{prop:level} Let $m \leq n$. \\
{\em (i)} The Hibi ring $\kk[\Pi_{m,n}]$ of $\Pi_{m,n}$ is level for any $m,n$. \\
{\em (ii)} The Hibi ring $\kk[\Pi_{m,n}']$ of $\Pi_{m,n}'$ is level if and only if $m=1$ or $m=2$. 
\end{prop}
\begin{proof}
The assertion (i) directly follows from Corollary~\ref{levelcor}. %and the facts that $\kk[K_{r_1,r_2}]\cong \kk[\Pi_{r_1-1,r_2-1}]$ and $\kk[K_{1,r_2,r_3}]\cong \kk[\Pi_{r_2,r_3}']$, 

Similarly, we see from Corollary~\ref{levelcor} that $\kk[\Pi_{1,n}']$ and $\kk[\Pi_{2,n}']$ are level. 

Let $y_1=p_{m+n}$ and $x_1=p_1$ and take the sequence $y_1,x_1$, which satisfies condition N. 
Then we see that 
$$r(y_1,x_1)=\rank [\hat{0},y_1]_{\widehat{\Pi}} - \rank [x_1,y_1]_{\widehat{\Pi}} + \rank [x_1,\hat{1}]_{\widehat{\Pi}}=n-1+m.$$ 
On the other hand, we have $\rank \widehat{\Pi}=n+1$. If $m \geq 3$, then $m-1>1$, so we have 
$$n+m-1=r(x_1,y_1)>\rank \widehat{\Pi}=n+1.$$ 
Hence, $\kk[\Pi_{m,n}']$ is not level when $m \geq 3$ by Theorem~\ref{levelthm}. Therefore, the assertion (ii) follows. 
\end{proof}
Hence, in the case where $n=2$ or $n=3$ with $r_1=1$, $\kk[K_{r_1,\ldots,r_n}]$ is level if and only if $K_{r_1,\ldots,r_n}$ satisfies (1) or (2) in Theorem~\ref{thm:level}. 

\medskip

Note that $\kk[K_{1,n}]$ is isomorphic to a polynomial ring with $n$ variables. 

Regarding the characterization of almost Gorenstein Hibi rings, see \cite[Introduction]{M18}. 
According to it, we see the following: 
\begin{itemize}
\item Let $m \leq n$. Consider the poset $\Pi_{m,n}$. 
\begin{itemize}
\item We see that $\Pi_{1,n}$ fits into the case of (1) of \cite[Introduction]{M18}. Thus, $\kk[\Pi_{1,n}]$ is almost Gorenstein. 
\item Since $\Pi_{m,m}$ is pure, we know that $\kk[\Pi_{m,m}]$ is Gorenstein, in particular, almost Gorenstein. 
\item If $1<m<n$, then we see from the characterization that $\kk[\Pi_{m,n}]$ is never almost Gorenstein. 
\end{itemize}
\item Let $m \leq n$. Consider the poset $\Pi_{m,n}'$. 
\begin{itemize}
\item We have $\star_{\Pi_{1,n}'}(p_{n+1})=\Pi_{1,n}'$ and $\Pi_{1,n}' \setminus \{p_{n+1}\}$ fits into the case of (1) of \cite[Introduction]{M18}. 
Thus, $\kk[\Pi_{1,n}]$ is almost Gorenstein. 
\item We see that $\Pi_{m,m}'$ fits into the case of (2) (ii) (with $p=0$). Hence, $\kk[\Pi_{m,m}']$ is almost Gorenstein. 
\item If $1<m< n$, then we see from the characterization that $\kk[\Pi_{m,n}']$ is never almost Gorenstein. 
\end{itemize}
\end{itemize}
%Therefore, in the case where $n=2$ or $n=3$ with $r_1=1$, $\kk[K_{r_1,\ldots,r_n}]$ is almost Gorenstein 
%if and only if $K_{r_1,\ldots,r_n}$ satisfies (1) or (2) in Theorem~\ref{thm:almGor}. 

\bigskip

%%%%%%%%%%%%%%%%%%%%%%%%%%%%%%%%%%%%%%%%%%%%%%%%%%%%%%%%%%%%%%%%%%%%%%%%%%%%%%%%%%%%%%%%%%%%%%%%%%%%%%%%%%%%%%%%%%%%%%%%%%%%%%%
%%%%%%%%%%%%%%%%%%%%%%%%%%%%%%%%%%%%%%%%%%%%%%%%%%%%%%%%%%%%%%%%%%%%%%%%%%%%%%%%%%%%%%%%%%%%%%%%%%%%%%%%%%%%%%%%%%%%%%%%%%%%%%%
%%%%%%%%%%%%%%%%%%%%%%%%%%%%%%%%%%%%%%%%%%%%%%%%%%%%%%%%%%%%%%%%%%%%%%%%%%%%%%%%%%%%%%%%%%%%%%%%%%%%%%%%%%%%%%%%%%%%%%%%%%%%%%%
%%%%%%%%%%%%%%%%%%%%%%%%%%%%%%%%%%%%%%%%%%%%%%%%%%%%%%%%%%%%%%%%%%%%%%%%%%%%%%%%%%%%%%%%%%%%%%%%%%%%%%%%%%%%%%%%%%%%%%%%%%%%%%%
%%%%%%%%%%%%%%%%%%%%%%%%%%%%%%%%%%%%%%%%%%%%%%%%%%%%%%%%%%%%%%%%%%%%%%%%%%%%%%%%%%%%%%%%%%%%%%%%%%%%%%%%%%%%%%%%%%%%%%%%%%%%%%%
%%%%%%%%%%%%%%%%%%%%%%%%%%%%%%%%%%%%%%%%%%%%%%%%%%%%%%%%%%%%%%%%%%%%%%%%%%%%%%%%%%%%%%%%%%%%%%%%%%%%%%%%%%%%%%%%%%%%%%%%%%%%%%%
%%%%%%%%%%%%%%%%%%%%%%%%%%%%%%%%%%%%%%%%%%%%%%%%%%%%%%%%%%%%%%%%%%%%%%%%%%%%%%%%%%%%%%%%%%%%%%%%%%%%%%%%%%%%%%%%%%%%%%%%%%%%%%%
%%%%%%%%%%%%%%%%%%%%%%%%%%%%%%%%%%%%%%%%%%%%%%%%%%%%%%%%%%%%%%%%%%%%%%%%%%%%%%%%%%%%%%%%%%%%%%%%%%%%%%%%%%%%%%%%%%%%%%%%%%%%%%%

\section{Proofs of main theorems}\label{sec:proof}

The goal of this section is to complete the proofs of Theorems~\ref{thm:level} and \ref{thm:almGor}. 
As shown in Section~\ref{sec:poset}, the case where $n=2$ or $n=3$ with $r_1=1$ was already done. 
Thus, in principle, we discuss the case where $n=3$ with $r_1 \geq 2$ or $n \geq 4$. 

\subsection{Preliminaries for edge polytopes}
Before proving the assertions, we recall some geometric information on edge polytopes.

For $i\in [d]$, let $p_i:\RR^d \to \RR$ be the $i$-th projection, and let $p_{V_k}:=\sum_{j\in V_k}p_j$ for $k\in [n]$.

For $k\in [n]$, let $\displaystyle f_k:=\sum_{i\in [d]\setminus V_k}\eb_i - \sum_{j\in V_k} \eb_j.$ 
We define $\Psi_r$, $\Psi_f$ and $\Psi$ as follows:
$$\Psi_r:=\{\eb_i:i\in [d]\}, \quad \Psi_f:=\{f_k:k\in [n]\}, \quad \Psi:=\Psi_r\cup \Psi_f.$$
Note that for $\eb_i \in \Psi_r$ (resp. $f_k\in \Psi_f$), we have $\langle -,\eb_i \rangle=p_i(-)$ 
(resp. $\langle -,f_k \rangle=\displaystyle \biggl( \sum_{j\in [n]\setminus \{k\}}p_{V_j} - p_{V_k}\biggr)(-)$), 
where $\langle -, - \rangle : \RR^d \times \RR^d \rightarrow \RR$ stands for the usual inner product. 

It is proved in \cite{OH98} that the hyperplanes $H_l:=\{x\in \RR^d:\langle x,l \rangle = 0\}$ for $l\in \Psi$ define the facets of $P_{K_{r_1,\ldots,r_n}}$. 
Although those are not necessarily one-to-one, in the case where $n = 3$ with $r_1\ge 2$ or $n\ge 4$, we see that it is one-to-one.

Let $G$ be a graph on the vertex set $[d]$. We say that $G$ satisfies \textit{odd cycle condition} 
if for any two distinct odd cycles which have no common vertex, there is a bridge between them. 
It is known that $\kk[G]$ is normal if and only if $G$ satisfies odd cycle condition. 
Note that $K_{r_1,\ldots,r_n}$ satisfies odd cycle condition. 
Let $(h_0,h_1,\ldots,h_s)$ be the $h$-vector of the edge ring $\kk[G]$, where $s$ is the scole degree, 
and let $\ell=\min\{m\in \ZZ_{>0}:mP_G^{\circ}\cap \ZZ^d\ne \emptyset\}$, 
where $P_G^{\circ}$ denotes the relative interior of $P_G$. 
Since the ideal generated by the monomials contained in the interior of $P_G$ is the canonical module of $\kk[G]$, 
we see that $\ell=-a(\kk[G])$. Thus, $d=\ell+s$ holds if $G$ is non-bipartite. 
In the case where $n = 3$ with $r_1\ge 2$ or $n\ge 4$, one has $\iota \in mP_{K_{r_1,\ldots,r_n}}^{\circ}\cap \ZZ^d$ 
if and only if $\langle \iota,l \rangle>0$ holds for all $l\in \Psi$.

For $\iota\in (\ell+k)P_G^\circ \cap \ZZ^d$ for $k\in \ZZ_{>0}$, 
then we say that $\iota$ is a \textit{first appearing interior point} in $(\ell+k)P_G^\circ \cap \ZZ^d$ 
if it cannot be written as a sum of $\iota'\in (\ell+i)P_G^\circ \cap \ZZ^d$ with $0\le i<k$ and $\rho(e)$'s with $e\in E(G)$, 
where the elements in $\ell P_G^\circ \cap \ZZ^d$ are regarded as first appearing interior points. 
Let $\mu_k(G)$ denote the number of first appearing interior points in $(\ell+k)P_G^\circ \cap \ZZ^d$ for $k \in \ZZ_{\ge 0}$. 
Note that $\mu_0(G)=h_s$ holds. 

\medskip

%Let $C$ be the Ulrich $\kk[K_{r_1,\ldots,r_n}]$-module and let $\mu(C)$ (resp. $e(C)$) denotes the number of elements in a minimal system of generators of $C$ as an $\kk[K_{r_1\ldots,r_n}]$ (resp. the multiplicity of $C$).

In what follows, for the study of $\kk[K_{r_1,\ldots,r_n}]$ with $1 \leq r_1 \leq \cdots \leq r_n$ and $d=\sum_{i=1}^dr_i$, 
we divide into the following three cases on $K_{r_1,\ldots,r_n}$: 
\begin{itemize}
\item[(A)] $2r_n<d$ and $d$ is even; 
\item[(B)] $2r_n<d$ and $d$ is odd; 
\item[(C)] $2r_n\ge d$. 
\end{itemize}

%%%%%%%%%%%%%%%%%%%%%%%%%%%%%%%%%%%%%%%%%%%%%%%%%%%%%%%

Given a graph $G$ with the edge set $E(G)$, we say that $\calM \subset E(G)$ is a \textit{perfect matching} (a.k.a. \textit{$1$-factor}) 
if every vertex of $G$ is incident to exactly one edge of $\calM$. 
\begin{lem}\label{matching}
The complete multipartite graph $K_{r_1,\ldots,r_n}$ has a perfect matching if and only if $d$ is even and $2r_n\le d$. 
\end{lem}
\begin{proof}
Tutte's theorem (see, e.g., \cite[Theorem 2.2.1]{Diestel}) claims that a graph $G$ on the vertex set $V(G)$ has a perfect matching 
if and only if $q(G - U) \leq |U|$ holds for any $U \subset V(G)$, 
where $q( \cdot )$ denotes the number of connected components with odd cardinality and 
$G - U$ denotes the induced subgraph of $G$ by $V(G) \setminus U$. 

When $U=V(K_{r_1,\ldots,r_n})$, the inequality trivially holds. 
When $U=\emptyset$, we can see that $q(G) = 0$ holds if and only if $d$ is even. 

Consider $\emptyset \neq U \subsetneq V(K_{r_1,\ldots,r_n})$. 
If there are two verticies $u,v$ in $V(K_{r_1,\ldots,r_n})\setminus U$ with $u \in V_i$ and $v \in V_j$ with $i \neq j$, 
the number of connected components of $K_{r_1,\ldots,r_n}-U$ is equal to $1$, so $q(K_{r_1,\ldots,r_n}-U)\le |U|$ holds. 
If there exists $k \in [n]$ with $V(K_{r_1,\ldots,r_n})\setminus U \subset V_k$, 
then it follows from $1\le r_1\le \cdots \le r_n$ and $V_i \subset U$ for $i\in [n]\setminus \{k\}$ that 
$q(K_{r_1,\ldots,r_n}-U)\le r_k \le r_n \le |U|$. Therefore, by considering the case $k=n$, we conclude the following: 
\begin{align*}
q(K_{r_1,\ldots,r_n}-U)\le |U|\ \text{for any }U &\Longleftrightarrow \ r_n\le \sum_{k\in [n-1]}r_k \text{ and $d$ is even } \\
&\Longleftrightarrow \ 2r_n\le d\text{ and $d$ is even.}
\end{align*}
\end{proof}

%%%%%%%%%%%%%%%%%%%%%%%%%%%%%%%%%%%%%%%%%%%%%%%%%%%%%%

\begin{prop}\label{prop}Let $(h_0,h_1,\ldots,h_s)$ be the $h$-vector of $\kk[K_{r_1,\ldots,r_n}]$.
\begin{itemize}
\item[(a)] In the case of (A), we have $\ell=d/2$ and $h_s=1$.
\item[(b)] In the case of (B), we have $\ell=(d+1)/2$ and $h_s\ge 2$.
\item[(c)] In the case of (C), we have $\ell=r_n+1$ and $h_s\ge 2$.
\end{itemize}
\end{prop}
\begin{proof}
In the case of (A), by Lemma~\ref{matching}, 
there exists a perfect matching $\calM \subset E(K_{r_1,\ldots,r_n})$, and we obtain $\rho(\calM):=\sum_{e\in \calM}\rho(e)=\sum_{i\in [d]}\eb_i$. 
We can see that $\rho(\calM)$ is the unique element in $|\calM|P_{K_{r_1,\ldots,r_n}}^\circ \cap \ZZ^d$ 
since $\langle \rho(\calM),l \rangle>0$ hold for all $l\in \Psi$, 
and it is clear that $mP_{K_{r_1,\ldots,r_n}}^\circ \cap \ZZ^d=\emptyset$ for $m<|\calM|$. Therefore, we have $\ell=|\calM|=d/2$ and $h_s=1$.

Next, assume the case of (B). We consider the induced subgraph $K_{r_1,\ldots,r_n}-\{v_n\}$ for $v_n\in V_n$. 
From the assumption, we observe that $d-1$ is even and $2\mathrm{min}\{r_{n-1},r_n-1\}\le 2r_n\le d-1$ holds. 
Hence, we can take a perfect matching $\calM'$ of $K_{r_1,\ldots,r_n}-\{v_n\}$ by Lemma~\ref{matching}. 
Take $v_1\in V_1$, add the edge $\{v_1,v_n\}$ to $\calM'$ and write $\calM''$ for it. 
Then we have $\rho(\calM'')=2\eb_{v_1}+\sum_{i\in [d]\setminus \{v_1\}}\eb_i$, and $\langle \rho(\calM''),l \rangle>0$ for all $l\in \Psi$. 
In fact, for $f_1$, we observe that 
$$\langle \rho(\calM''),f_1 \rangle=\biggl( \sum_{i\in [n]\setminus \{1\}}p_{V_i} - p_{V_1}\biggr) (\rho(\calM''))=\displaystyle \sum_{k\in [n]\setminus \{1\}}r_k-(r_1+1)\ge \sum_{k\in [n]\setminus \{1,2\}}r_k-1>0.$$
Thus, we have $\rho(\calM'')\in |\calM''|P_{K_{r_1,\ldots,r_n}}^\circ \cap \ZZ^d$ and $mP_{K_{r_1,\ldots,r_n}}^\circ \cap \ZZ^d=\emptyset$ for $m<|\calM''|$. 
Moreover, by exchanging $v_1$ with another vertex of $V_1$ or a vertex of $V_2$ if $r_1=1$, 
we obtain $\rho(\calM'')\in |\calM''|P_{K_{r_1,\ldots,r_n}}^\circ \cap \ZZ^d$ in the same way. 
Therefore, we have $\ell=|\calM''|=(d+1)/2$ and $h_s\ge 2$.

Finally, assume the case of (C). Let $r_n':=\sum_{i\in [n-1]}r_i$ and let $r_n'':=r_n-r_n'$. 
We join $r_n'$ verticies of $V_n$ to verticies of $[d]\setminus V_n$ one-by-one, and join the remaining $r_n''$ verticies of $V_n$ to $v_1$, 
and join a vertex $v_2\in V_2$ to $v_1$. Let $\calE$ be the set of those edges. 
Then we have $\rho(\calE)=(r_n''+2)\eb_{v_1}+2\eb_{v_2}+\sum_{i\in [d]\setminus \{v_1,v_2\}}\eb_i$, and $\langle \rho(\calE),l \rangle>0$ for all $l\in \Psi$. 
In fact, for $f_1$, we observe that 
$$\langle \rho(\calE),f_1 \rangle=\displaystyle \sum_{k\in [n]\setminus \{1\}}r_k+1-(r_1+r_n''+1)\ge \sum_{k\in [n]\setminus \{1,2,n\}}r_k+r_n'>0.$$
Thus, we have $\rho(\calE)\in |\calE|P_{K_{r_1,\ldots,r_n}}^\circ \cap \ZZ^d$ and $mP_{K_{r_1,\ldots,r_n}}^\circ \cap \ZZ^d=\emptyset$ for $m<|\calE|$ 
since $\langle \iota,r \rangle>0$ for all $\iota \in \ell P_{K_{r_1,\ldots,r_n}}^\circ \cap \ZZ^d$ and $r\in \Psi$. %we need have $a\ge r_n$. Moreover, in order to satisfy $\langle \iota,f_n \rangle>0$, we need an edge whose endverticies are not in $V_n$.
Moreover, by exchanging $v_2$ with another vertex of $V_2$ or a vertex of $V_3$ if $r_2=1$, 
we obtain $\rho(\calE)\in |\calE|P_{K_{r_1,\ldots,r_n}}^\circ \cap \ZZ^d$ in the same way. Therefore, we have $\ell=|\calE|=r_n+1$ and $h_s\ge 2$.
\end{proof}

\begin{cor}\label{cor}
%\begin{itemize}
%\item[(i)] If $K_{r_1,\ldots,r_n}$ satisfies (B) or (C), then $\kk[K_{r_1,\ldots,r_n}]$ is not almost Gorenstein.
%\item[(ii)] $\kk[K_{1,1,1,r_4}]$ and $\kk[K_{1,1,1,1,1}]$ are level.
%\end{itemize}
Assume that $n = 3$ with $r_1 \geq 2$ or $n \geq 4$. 
If $\kk[K_{r_1,\ldots,r_n}]$ is almost Gorenstein, then $K_{r_1,\ldots,r_n}$ is in the case of (A). 
\end{cor}
\begin{proof}
In our assumption, we have $d \geq 4$. Note that $s=d-\ell$. 

In the case of (B), since $s=d-\ell=d-(d+1)/2=(d-1)/2$ by Proposition~\ref{prop} (b), and $d$ is odd, we see that $s \geq 2$. 
In the case of (C), since $s=d-\ell=\sum_{i \in [n]}r_i -r_n-1=\sum_{i \in [n-1]}r_i -1$ by Proposition~\ref{prop} (c), we see that $s \geq 2$. 
Moreover, in both cases, we also know that $h_s \geq 2$. Therefore, these are never almost Gorenstein by Theorem~\ref{thm:h_s=1}. 
%(ii) We can see that $\kk[K_{1,1,1,1}]$ is Gorenstein, and $\kk[K_{1,1,1,2}]$ and $\kk[K_{1,1,1,1,1}]$ are level by using Macaulay2. From Proposition~\ref{prop} (c), $K_{1,1,1,r_4}$ with $r_4\ge 3$ satisfies condition (C) and $s=2$. It is always level. 
\end{proof}

\medskip
%%%%%%%%%%%%%%%%%%%%%%%%%%%%%%%%%%%%%%%%%%%%%%%%%%%%%%%%%%%%%%%%%

\subsection{Proof of Theorem~\ref{thm:level}}

This subsection is devoted to proving Theorem~\ref{thm:level}. 
Since the case where $n=2$ or $n = 3$ with $r_1=1$ has been already done in Section~\ref{sec:poset}, 
we assume that $n =3$ with $r_1 \geq 2$ or $n \geq 4$. 
Under this assumption, we prove that $\kk[K_{r_1,\ldots,r_n}]$ is level if and only if one of the following holds: 
\begin{equation}\label{condition:level}
\begin{split}
&\text{$n=3$ with $r_1=r_2=2$ or $r_1=r_2=r_3=3$;} \\
&\text{$n=4$ with $r_1=r_2=r_3=1$;} \\
&\text{$n=5$ with $r_1=\cdots=r_5=1$}. 
\end{split}
\end{equation}

Assume the case of (A). 
By Proposition~\ref{prop}, we have $h_s=1$. Thus, $\kk[K_{r_1,\ldots,r_n}]$ is Gorenstein if it is level (see Remark~\ref{rem:level_Gor}). 
Hence, $\kk[K_{r_1,\ldots,r_n}]$ is level if and only if $K_{r_1,\ldots,r_n}=K_{2,2,2}$ or $K_{1,1,1,1}$ by \cite[Remark 2.8]{OH00}. 

Therefore, in what follows, we consider the cases (B) and (C). 

\medskip

\noindent
{\bf ``Only if'' part}: 

Assume the case of (B). Take $\calM''$ and $v_1\in V_1$ as in the proof of Proposition~\ref{prop}. 
Then there exsists an edge $\{i,j\}\in \calM''$ such that $i \not\in V_1$ and $j \not\in V_1$. 
Remove such edge from $\calM''$ and add $\{v_1,i\}$ and $\{v_1,j\}$ to $\calM''$. 
Write $\calN$ for it. Then we have $$\rho(\calN)=4\eb_{v_1}+\sum_{i\in [d]\setminus \{v_1\}}\eb_i \in (\ell+1)P_{K_{r_1,\ldots,r_n}}\cap \ZZ^d.$$ 
Since there is only one entry which is more than $1$, we see that $\rho(\calN)$ cannot be written 
as a sum of $\ell P_{K_{r_1,\ldots,r_n}}^\circ \cap \ZZ^d$ and $\rho(e)$ for $e\in E(K_{r_1,\ldots,r_n})$. 
Hence, once we have $\rho(\calN)\in (\ell+1)P_{K_{r_1,\ldots,r_n}}^\circ \cap \ZZ^d$, it is not level. 
Since we know $\langle \rho(\calN),r \rangle>0$ for all $r\in \Psi_r$, we may observe those of $\Psi_f$: 
\begin{align}
&\langle f_1,\rho(\calN)\rangle =\sum_{i\in [n]\setminus \{1\}}r_i-(r_1+3)>0 \label{level(B)f_1};  \\
&\langle f_k,\rho(\calN)\rangle =\sum_{i\in [n]\setminus \{k\}}r_i+3-r_k>0\ \text{for}\ k\in [n]\setminus \{1\}.\label{level(B)f_k} 
\end{align}
The inequality \eqref{level(B)f_k} always holds by the assumption (B). The inequality \eqref{level(B)f_1} holds if 
\begin{align*}
n &\ge 6,  \\
n &=5\ \text{with}\ r_5\ge 2,  \\
n &=4\ \text{with}\ r_4\ge 3, \text{ or} \\
n &=3\ \text{with}\ r_3\ge 4. 
\end{align*} 
Therefore, in the case of (B), 
\begin{align*}
\text{$\kk[K_{r_1,\ldots,r_n}]$ are not level except for $K_{1,1,1,1,1}$, $K_{1,1,1,2}$, $K_{2,2,3}$, and $K_{3,3,3}$.} 
\end{align*} 
Note that we can confirm that $\kk[K_{1,2,2,2}]$ is not level by using {\tt Macaulay2} (\cite{M2}). 

\medskip

Assume the case of (C). Take $\calE$, $v_1\in V_1$, and $v_2\in V_2$ as in the proof of Proposition~\ref{prop}. 
In $\calE$, let $v_n\in V_n$ be the vertex adjacent to $v_2$, 
and let $v_n'$ be the vertex adjacent to a vertex $v_2'\ne v_2$ of $V_2$ or a vertex $v_3\in V_3$ if $r_2=1$. 
Remove $\{v_2,v_n\}$ and $\{v_2',v_n'\}$ from $\calE$ and add $\{v_1,v_n\}$, $\{v_1,v_2'\}$ and $\{v_1,v_n'\}$ to $\calE$. 
Write $\calN'$ for it. Then we have 
$$\rho(\calN')=(r_n''+5)\eb_{v_1}+\sum_{i\in [d]\setminus \{v_1\}}\eb_i \in (\ell+1)P_{K_{r_1,\ldots,r_n}}\cap \ZZ^d.$$ 
Then we see that $\rho(\calN')$ cannot be written 
as a sum of $\ell P_{K_{r_1,\ldots,r_n}}^\circ \cap \ZZ^d$ and $\rho(e)$ for $e\in E(K_{r_1,\ldots,r_n})$. 
Hence, once we have $\rho(\calN')\in (\ell+1)P_{K_{r_1,\ldots,r_n}}^\circ \cap \ZZ^d$, it is not level. 
Since we know $\langle \rho(\calN'),r \rangle>0$ for all $r\in \Psi_r$, we may observe those of $\Psi_f$: 
\begin{align}
&\langle f_1,\rho(\calN')\rangle =\sum_{i\in [n]\setminus \{1\}}r_i-(r_1+r_n''+4)>0; \label{level(C)f_1} \\
&\langle f_k,\rho(\calN')\rangle =\sum_{i\in [n]\setminus \{k\}}r_i+r_n''+4-r_k>0\ \text{for}\ k\in [n]\setminus \{1\}.\label{level(C)f_k}
\end{align}
The inequality \eqref{level(C)f_k} always holds by (C). The inequality \eqref{level(C)f_1} holds if 
\begin{align*}
n &\ge 5, \\
n &=4\ \text{with}\ r_3\ge 2, \text{ or} \\
n &=3\ \text{with}\ r_2\ge 3.
\end{align*}
Thus, in the case of (C), 
\begin{align*}
\text{$\kk[K_{r_1,\ldots,r_n}]$ is not level except for $K_{2,2,r_3}$ with $r_3\ge 4$ and $K_{1,1,1,r_4}$ with $r_4\ge 3$.} 
\end{align*}

Therefore, we obtain that $\kk[K_{r_1,\ldots,r_n}]$ is not level if not in the case \eqref{condition:level}. 
%We can see that it follows from Lemma~\ref{levellemma} that $\kk[K_{r_1,\ldots,r_n}]$ is level if $K_{r_1,\ldots,r_n}$ satisfies (l2), and from Corollary~\ref{cor} (ii) that $\kk[K_{r_1,\ldots,r_n}]$ is level if $K_{r_1,\ldots,r_n}$ satisfies (l3).

\medskip

\noindent
{\bf ``If'' part}: 

Our remaining task is to show that the edge rings of \eqref{condition:level} are level. 
%$\kk[K_{2,2,r_3}]$ with $r_3\ge 2$ and $\kk[K_{3,3,3}]$ are level. 

\noindent
($K_{2,2,r_3}$ with $r_3\ge 2$)

If $r_3=2$, $\kk[K_{2,2,2}]$ is Gorenstein, and if $r_3=3$, $\kk[K_{2,2,3}]$ is level by using {\tt Macaulay2}. 

Hence, let us assume that $r_3\ge 4$. Then $K_{r_1,\ldots,r_n}$ satisfies (C). 
Thus, we have $\ell=r_3+1$. It is enough to show that for any $k \geq 0$ and $\iota \in (\ell+k)P_{K_{2,2,r_3}}^\circ \cap \ZZ^d$, 
$\iota$ can be written as a sum of an element of $\ell P_{K_{2,2,r_3}}^\circ \cap \ZZ^d$ and $k$ elements of $P_{K_{2,2,r_3}} \cap \ZZ^d$, 
i.e., $\rho(e_1),\ldots,\rho(e_k)$ with $e_1,\ldots,e_k\in E(K_{2,2,r_3})$. We show this by induction on $k$. The case $k=0$ trivially holds. 

We have $\biggl( \sum_{i\in [d]}p_i\biggr)(\iota)=2(\ell+k)=2r_3+2k+2\ge 2r_3+4$, $p_i(\iota)>0$ for $i\in [d]$, 
$p_{V_1}(\iota)$, $p_{V_2}(\iota)\ge 2$, and $p_{V_3}(\iota)\ge r_3$. 
In the case $p_{V_3}(\iota)=r_3$, we can see that $\langle \iota,f_j\rangle>0$ holds, that is, 
$p_{V_j}(\iota)\ge 3$ for $j=1,2$, and there exsist a $v_1\in V_1$ and a $v_2\in V_2$ such that $p_{v_j}(\iota)\ge 2$ for $j=1,2$. 
Let $\iota':=\iota-\rho(\{v_1,v_2\})$. If $\langle \iota',l \rangle>0$ holds for $l\in \Psi$, 
we have $\iota'\in (\ell+k-1)P_{K_{2,2,r_3}}^\circ \cap \ZZ^d$. It is enough to discuss that of $f_3$: 
$$\langle \iota',f_3\rangle =\biggl(\sum_{k\in \{1,2\}}p_{V_k}\biggr)(\iota)-2-p_{V_3}(\iota)=(r_3+2k+2)-2-r_3>0.$$

In the case $p_{V_3}(\iota)\ge r_3+1$, there exists a $v_3\in V_3$ such that $p_{v_3}(\iota)=2$. 
We may assume that $p_{V_1}(\iota)\le p_{V_2}(\iota)$. Then there is a $v_2'\in V_2$ such that $p_{v_2'}(\iota)\ge 2$. 
Let $\iota':=\iota-\rho(\{v_2',v_3\})$. If we have $\langle \iota,l \rangle>0$ for $l\in \Psi$,
we obtain $\iota'\in (\ell+k-1)P_{K_{2,2,r_3}}^\circ \cap \ZZ^d$. It is enough to discuss that of $f_1$: 
$$\langle \iota',f_1\rangle =\biggl(\sum_{k\in \{2,3\}}p_{V_k}\biggr)(\iota)-2-p_{V_1}(\iota)\ge p_{V_3}(\iota)-2>0.$$

Therefore, we obtain the desired result. 

\smallskip

\noindent
($K_{3,3,3}$)

In the same way as above, it is enough to show that for any $k \geq 0$ and $\iota \in (\ell+k)P_{K_{3,3,3}}^\circ \cap \ZZ^d$, 
$\iota$ can be written as a sum of an element of $\ell P_{K_{3,3,3}}^\circ \cap \ZZ^d$ and $\rho(e_1),\ldots,\rho(e_k)$ with $e_1,\ldots,e_k\in E(K_{3,3,3})$. 

By $\ell=5$, we have $\biggl( \sum_{i\in [d]}p_i\biggr)(\iota)=2(\ell+k)=2k+10\ge 12$. 
We may assume that $p_{V_1}(\iota)\le p_{V_2}(\iota)\le p_{V_3}(\iota)$. 
If we have $p_{V_1}(\iota)=p_{V_2}(\iota)=3$, we obtain $p_{V_3}(\iota)=2k+4\ge 6$ and $\langle \iota',f_3 \rangle \le 0$. This is a contradiction. 
Thus, we have $4\le p_{V_2}(\iota)\le p_{V_3}(\iota)$. Hence, there exsist a $v_2\in V_2$ and $v_3\in V_3$ 
such that $p_{v_j}(\iota)\ge 2$ for $j=1,2$. Let $\iota':=\iota-\rho(\{v_2,v_3\})$. 
If $\langle \iota',l \rangle>0$ holds for all $\l\in \Psi$, we obtain $\iota'\in (\ell+k-1)P_{K_{3,3,3}}^\circ \cap \ZZ^d$. 
It is enough to discuss that of $f_1$: 
$$\langle \iota',f_1\rangle =\biggl(\sum_{k\in \{2,3\}}p_{V_k}\biggr)(\iota)-2-p_{V_1}(\iota)=\Bigl(p_{V_2}-p_{V_1}\Bigr)(\iota)+\Bigl(p_{V_3}(\iota)-2\Bigr)>0.$$ 

Therefore, we obtain the desired result. 

\smallskip

\noindent
($K_{1,1,1,r_4}$ with $r_4 \geq 1$)

We can see that $\kk[K_{1,1,1,1}]$ is Gorenstein, and we can check by {\tt Macaulay2} that $\kk[K_{1,1,1,2}]$ is level. 
For $r_4 \geq 3$, by Proposition~\ref{prop} (c), $K_{1,1,1,r_4}$ is in the case of (C) and $s=d-r_4-1=2$. 
It is always level by Theorem~\ref{thm:Y}. 

\noindent
($\kk[K_{1,1,1,1,1}]$) 

We can check by {\tt Macaulay2} that $\kk[K_{1,1,1,1,1}]$ is level. 

\medskip

%%%%%%%%%%%%%%%%%%%%%%%%%%%%%%%%%%%%%%%%%%%%%%%%%%%%%%%%%%%%%%%%%

\subsection{Proof of Theorem~\ref{thm:almGor}}

We still assume the condition $n=3$ with $r_1\ge 2$ or $n\ge 4$. 
This subsection is devoted to giving a proof of Theorem~\ref{thm:almGor}.

We recall a notion of Ehrhart polynomials. 
Let $P \subset \RR^N$ be an integral convex polytope, which is a convex polytope all of whose vertices belong to $\ZZ^N$. 
For $m \in \ZZ_{>0}$, consider the number of integer points contained in $mP \cap \ZZ^N$. 
Then it is known that such number $|mP\cap \ZZ^N|$ can be described by a polynomial in $m$ of degree $\dim P$, denoted by $i(P,m)$. 
The enumerating polynomial $i(P,m)$ is called the {\em Ehrhart polynomial} of $P$. 
For the introduction to the Ehrhart polynomials, see, e.g., \cite{BR}. 

Throughout the remaining parts of this section, let $R=\kk[K_{r_1,\ldots,r_n}]$. 
Note that $R$ is normal since $K_{r_1,\ldots,r_n}$ satisfies odd cycle condition. 
Regarding the definition of almost Gorensteinness, let $C$ be the cokernel of the injection $R \to \omega_R(-a)$. 
Note that $C$ is a Cohen--Macaulay $R$-module of dimension $d-1$. Our goal is to characterize when $e(C)=\mu(C)$ holds. 
For this, we prepare the following two lemmas. 
\begin{lem}\label{multiplicity}
Assume the case of (A). Then 
%Assume that $K_{r_1,\ldots,r_n}$ satisfies (A). Then
\begin{align*}%\label{eq:e(C)}
e(C)=\displaystyle \sum_{k\in [n]} \Bigl(\frac{d}{2}-r_k-1\Bigr)\binom{d-2}{r_k-1}.
\end{align*}
\end{lem}
\begin{proof}
Let $i(P_{K_{r_1,\ldots,r_n}},m)=c_{d-1}m^{d-1}+c_{d-2}m^{d-2}+\cdots+1$ be the Ehrhart polynomial of $P_{K_{r_1,\ldots,r_n}}$. 
Since $R$ is normal, we see that $H(R,m)=i(P_{K_{r_1,\ldots,r_n}},m)$. 
Note that $i(P_{K_{r_1,\ldots,r_n}}^{\circ},m)=(-1)^{d-1}i(P_{K_{r_1,\ldots,r_n}},-m)$ holds. (See, e.g., \cite[Theorem 4.1]{BR}.) 
From an exact sequence \eqref{eq:exact}, we can see that the Hilbert function $H(C,m)$ of $C$ coincides with 
$$i(P_{K_{r_1,\ldots,r_n}}^{\circ},m+\ell)-i(P_{K_{r_1,\ldots,r_n}},m),$$ 
where $\ell=-a(R)$. This implies that the leading coefficient of $H(C,m)$ coincides with $(d-1)\ell c_{d-1}-2c_{d-2}$. 
Note that $\dim C=d-1$. Thus, $$e(C)=(d-2)!((d-1)\ell c_{d-1}-2c_{d-2}).$$ 
Here, \cite[Theorem 2.6]{OH00} claims that 
$$i(P_{K_{r_1,\ldots,r_n}},m)=\displaystyle \binom{d+2m-1}{d-1}-\sum_{k\in [n]} \sum_{1\le i\le j\le r_k}\binom{j-i+m-1}{j-i}\binom{d-j+m-1}{d-j}.$$ 
Hence, a direct computation shows that 
\begin{align*}
&(d-1)!c_{d-1}=\displaystyle 2^{d-1}-\sum_{k\in [n]} \sum_{j\in [r_k]}\binom{d-1}{j-1} \;\text{ (see \cite[Corollary 2.7]{OH00}), and} \\
&(d-2)!c_{d-2}=\displaystyle 2^{d-3}d-\sum_{k\in [n]}\sum_{j\in [r_k]}\Biggl( \binom{d-2}{j-2}+\frac{d-2}{2}\biggl( \binom{d-3}{j-3}+\binom{d-3}{j-1}\biggr) \Biggr).
\end{align*}
Remark $\ell=d/2$ by Proposition~\ref{prop} (a). Therefore, we conclude that 
$$e(C)=\displaystyle \sum_{k\in [n]}\sum_{j\in [r_k]}\Biggl( 2\binom{d-2}{j-2}+(d-2)\biggl( \binom{d-3}{j-3}+\binom{d-3}{j-1}\biggr)-\frac{d}{2}\binom{d-1}{j-1} \Biggr),$$
where we set $\binom{n}{r}$ for $r\in \ZZ_{<0}$ to be 0. By using 
\begin{align}\displaystyle \binom{n-1}{r-1}+\binom{n-1}{r}=\binom{n}{r}\ \text{for}\ n,r\in \ZZ, \label{bino}\end{align}
we obtain that 
$$e(C)=\displaystyle \sum_{k\in [n]}\sum_{j\in [r_k]}\Biggl( \frac{d}{2}\binom{d-1}{j-1}-2(d-1)\binom{d-2}{j-1}+2(d-2)\binom{d-3}{j-1}\Biggr).$$

It is enough to prove that
\begin{align}\label{induction}
\displaystyle \sum_{j\in [r_k]}\Biggl( \frac{d}{2}\binom{d-1}{j-1}-2(d-1)\binom{d-2}{j-1}+2(d-2)\binom{d-3}{j-1}\Biggr)=\Bigl(\frac{d}{2}-r_k-1\Bigr)\binom{d-2}{r_k-1}
\end{align}
holds. We prove this by induction on $r_k$. We can directly see that \eqref{induction} holds when $r_k=1$. 

Suppose that $r_k>1$. By the hypothesis of induction, we have 
\begin{align*}
\displaystyle \sum_{j\in [r_k+1]}&\Biggl( \frac{d}{2}\binom{d-1}{j-1}-2(d-1)\binom{d-2}{j-1}+2(d-2)\binom{d-3}{j-1}\Biggr) \\
=&\Bigl(\frac{d}{2}-r_k-1\Bigr)\binom{d-2}{r_k-1}+\Biggl(\frac{d}{2}\binom{d-1}{r_k}-2(d-1)\binom{d-2}{r_k}+2(d-2)\binom{d-3}{r_k}\Biggr). \\
\intertext{By using $\displaystyle (d-r_k-p)\binom{d-p}{r_k}=(d-p)\binom{d-p-1}{r_k}$ for $p=1,2$ and \eqref{bino}, we obtain that}
&\Bigl(\frac{d}{2}-r_k-1\Bigr)\biggl(\binom{d-1}{r_k}-\binom{d-2}{r_k}\biggr) \\
&+ \Biggl(\frac{d}{2}\binom{d-1}{r_k}-2(d-1)\binom{d-2}{r_k}+2(d-r_k-2)\binom{d-2}{r_k}\Biggr) \\
=&\Bigl(d-r_k-1\Bigr)\binom{d-1}{r_k}-\Bigl(\frac{d}{2}+r_k+1\Bigr)\binom{d-2}{r_k} \\
=&\bigl(d-1\bigr)\binom{d-2}{r_k}-\Bigl(\frac{d}{2}+r_k+1\Bigr)\binom{d-2}{r_k}=\Bigl(\frac{d}{2}-r_k-2\Bigr)\binom{d-2}{r_k}.
\end{align*}
This completes the proof. 
\end{proof}

\begin{lem}\label{m.s.o.g}
Assume the case of (A). Then 
$$\mu(C)=\displaystyle \sum_{k\in [n]}\sum_{j\in [\frac{d}{2}-r_k-1]}\binom{r_k-1+2j}{r_k-1},$$
where we let $\displaystyle \sum_{j\in [\frac{d}{2}-r_k-1]}\binom{r_k-1+2j}{r_k-1}=0$ if $\displaystyle r_k=\frac{d}{2}-1$. 
\end{lem}

\begin{proof}
Remark $\ell=d/2$. 

Since $\mu(C)=\sum_{j\ge \ell}\mu_j(\omega_{K_{r_1,\ldots,r_n}})-1=\sum_{j\ge \ell+1}\mu_j(\omega_{K_{r_1,\ldots,r_n}})=\sum_{j \geq 1}\mu_j(K_{r_1,\ldots,r_n})$, 
where we recall that $\mu_j(G)$ is the number of first appearing interior points. We compute $\mu_j(K_{r_1,\ldots,r_n})$ for $j\ge 1$. 

From Lemma~\ref{interior} below, we see that $\iota\in (\ell+j)P_{K_{r_1,\ldots,r_n}}\cap \ZZ^d$ ($j\ge 1$) is a first appearing interior point 
if and only if there exists $k\in [n]$ such that $\iota$ satisfies 
$$\begin{cases}
p_{V_j}(\iota)=r_k+2j, \\
p_i(\iota)=1\ \text{for}\ i\in [d]\setminus V_k, \text{ and }\\
\langle \iota,f_k \rangle=(d-r_k)-(r_k+2j)>0, \;\text{that is},\; 1\le j \le d/2-r_k-1.
\end{cases}$$
Hence, for $j$ and $k$ respectively, we observe that the number of first appearing interior points is $\displaystyle \binom{r_k-1+2j}{r_k-1}$, 
and so $\mu(C)=\displaystyle \sum_{j\ge 1}\mu_j(K_{r_1,\ldots,r_n})=\sum_{j\ge 1}\sum_{k\in [n]}\binom{r_k-1+2j}{r_k-1}.$ 
\end{proof}

\begin{lem}\label{interior}
Assume the case of (A). Given $\iota\in (\ell+j)P_{K_{r_1,\ldots,r_n}}^\circ \cap \ZZ^d$ for each $j\ge 0$, 
$\iota$ is a first appearing interior point if and only if there exists $k\in [n]$ such that 
\begin{align}\label{cases}
\begin{cases}
p_{V_j}(\iota)=r_k+2j, \\
p_i(\iota)=1\ \text{for}\ i\in [d]\setminus V_k.
\end{cases}
\end{align}
\end{lem}
\begin{proof}
{\bf ``If'' part}: By the condition on $\iota$, we see that $\iota$ cannot be written 
as a sum of an element of $(\ell+j')P_{K_{r_1,\ldots,r_n}}^\circ \cap \ZZ^d$ with $j'<j$ and 
$\rho(e_1),\ldots,\rho(e_{j'})$ with $e_1,\ldots,e_{j'}\in E(K_{r_1,\ldots,r_n})$. Thus, we obtain the desired result. 

\smallskip

\noindent
{\bf ``Only if'' part}:  We prove the assertion by induction on $j\ge 0$. 
If $j=0$, then $\sum_{i\in [d]}\eb_i$ is the unique first appearing interior point in $\ell P_{K_{r_1,\ldots,r_n}}^\circ \cap \ZZ^d$. 

Let $j\ge 1$. Let $r_k':=p_{V_k}(\iota)-r_k$ for $k\in [n]$. By the hypothesis of induction, 
there are at least two $k$'s with $r_k'\ne 0$. Take these $r_{k_1}'\ge r_{k_2}'\ge \cdots \ge r_{k_s}'>0$ and $k_p>k_q$ if $r_{k_p}'=r_{k_q}'$. 
Remark $s\ge 2$. Then there are $v_{k_1}\in V_{k_1}$ and $v_{k_2}\in V_{k_2}$ such that $p_{v_{k_1}},p_{v_{k_2}}\ge 2$. 

If we can have $\iota'=\iota-\rho(\{v_{k_1},v_{k_2}\}) \in (\ell+j-1)P_{K_{r_1,\ldots,r_n}}^\circ \cap \ZZ^d$, 
then $\iota$ is not a first appearing interior point. 
Since $\langle \iota',r\rangle>0$ for $r\in \Psi_r$ holds, we may show that $\langle \iota',f_k\rangle>0$ with $k=\max\{p_{V_j}(\iota):j\in [n]\}$. 
We see that $k$ should be one of $k_1$, $k_3$ and $n$.

\begin{itemize}
\item[($k=k_1$)] We have $\langle \iota',f_{k_1}\rangle=\langle \iota,f_{k_1}\rangle>0$. 
\item[($k=k_3$)] We see that $p_{V_{k_1}}(\iota)=p_{V_{k_2}}(\iota)=p_{V_{k_3}}(\iota)$. Remark $p_{V_k}(\iota)=r_k+r_k'$. Then 
\begin{align*}
\langle \iota',f_{k_3}\rangle&=\biggl( \sum_{i\in [n]\setminus \{k_3\}}p_{V_i}\biggr)(\iota)-2 - p_{V_{k_3}}(\iota) \\
&=\biggl( \sum_{i\in [n]\setminus \{k_1,k_2,k_3\}}p_{V_i}\biggr)(\iota)+(r_{k_2}-1)+(r_{k_2}'-1)>0.
\end{align*}
\item[($k=n$)] If $r_n'\ge r_{k_2}'$, then we have $n=k_1$ or $k_2$, so we can deduce the case $k=k_1$ or $k_3$. 
Hence, we may assume that $r_n'<r_{k_2}'$. Then we see that 
\begin{align*}
\langle \iota',f_n\rangle&=\biggl( \sum_{i\in [n-1]}p_{V_i}\biggr)(\iota)-2 - p_{V_n}(\iota) \\
&\ge \biggl( \sum_{i\in [n-1]}r_i-r_n\biggr)+(r_{k_1}'-1)+(r_{k_2}'-r_n'-1)>0.
\end{align*}
\end{itemize}
Therefore, we obtain the desired result.
\end{proof}

\begin{lem}\label{lem:rephrase}
Let $r_1,\ldots,r_n$ be positive integers with $1 \leq r_1 \leq \cdots \leq r_n$, let $d=\sum_{i \in [n]}r_i$ be even. 
Assume that $n = 3$ with $r_1 \geq 2$ or $n \geq 4$. 
Then $n$ and $(r_1,\ldots,r_n)$ satisfy one of the conditions (iii)---(vi) in Theorem~\ref{thm:almGor} 
if and only if $r_i \in \{1,d/2-1\}$ holds for any $i \in [n]$. 
\end{lem}
\begin{proof}
Since ``only if'' part is easy to see, we prove ``if'' part. 

Assume that $r_i \in \{1,d/2-1\}$ holds for any $i \in [n]$. Note that $d \geq n$ by definition. 
Let $\alpha$ be the number of $r_i$'s with $r_i=d/2-1$. Then $r_1=\cdots=r_{n-\alpha}=1$. 
Thus, we have $d=(n-\alpha)+(d/2-1)\alpha$. 

The case $\alpha=0$ is nothing but the case (vi). Note that $n$ should be even by $d=n$. 
If $\alpha=1$, then $d=(n-1)+(d/2-1)$ holds by definition, which implies that $d=2n-4$. Thus, this corresponds to the case (v). 

Suppose $\alpha \geq 2$. When $n \geq 5$, we see that $d = (d/2-2)\alpha + n \geq d-4 + n > d$, 
a contradiction. Hence, $n =3$ or $n=4$. 
\begin{itemize}
\item Let $n=3$. Since we assume $r_1 \geq 2$, we may discuss the case $\alpha=3$. Then $d=3(d/2-1)$ holds, i.e., $d=6$. 
Hence, we obtain that $r_1=r_2=r_3=2$, which is the case (iii). 
\item Let $n=4$. If $\alpha=2$, then we see that $r_3(=r_4)$ can be arbitrary, which is the case (iv). 
If $\alpha=3$ (resp. $\alpha=4$), then $d=1+3(d/2-1)$ (resp. $d=4(d/2-1)$) holds, i.e., $d=4$. 
Hence, we obtain that $r_1=\cdots=r_4=1$, which is included in (iv). 
\end{itemize}
\end{proof}

Now, we are ready to give a proof of Theorem~\ref{thm:almGor}. Since the case where $n=2$ or $n=3$ with $r_1 =1$ has been already done in Section~\ref{sec:poset}, 
we assume that $n=3$ with $r_1 \geq 2$ or $n \geq 4$. Under this assumption, 
it suffices to prove that $\kk[K_{r_1,\ldots,r_n}]$ is almost Gorenstein if and only if $d$ is even and $r_i \in \{1, d/2-1\}$ for any $i\in [n]$ by Lemma~\ref{lem:rephrase}. 
Then we may assume the case (A) by Corollary~\ref{cor}. Remark that $1\le r_k \le d/2-1$ holds. 
By Lemmas~\ref{multiplicity} and \ref{m.s.o.g}, we have 
$e(C)=\displaystyle \sum_{k\in [n]}e_k(C)\ \text{and}\ \mu(C)=\displaystyle \sum_{k\in [n]}\mu_k(C)$, where \begin{align*}
e_k(C):=\Bigl(\frac{d}{2}-r_k-1\Bigr)\binom{d-2}{r_k-1}\;\text{ and }\;\mu_k(C)=\sum_{j\in [\frac{d}{2}-r_k-1]}\binom{r_k-1+2j}{r_k-1} \;\text{ for each }k\in [n]. 
\end{align*}
\begin{itemize}
\item If $r_k=1$, then we have $e_k(C)=\mu_k(C)=d/2-2$. 
\item If $r_k=d/2-1$, we have $e_k(C)=\mu_k(C)=0$. 
\item If $1<r_k<d/2-1$, then we have 
$$\mu_k(C)\le \sum_{j\in [\frac{d}{2}-r_k-1]}\binom{d-r_k-3}{r_k-1}=\left(\frac{d}{2}-r_k-1\right)\binom{d-r_k-3}{r_k-1}<e_k(C).$$
\end{itemize}
Hence, if there is $k\in [n]$ with $1<r_k<d/2-1$, then $e(C)>\mu(C)$ holds. 
Therefore, we conclude that $e(C)=\mu(C)$ holds if and only if $r_k=1$ or $d/2-1$ for any $i \in [n]$.


\bigskip


\begin{thebibliography}{99}
\bibitem{BR} M. Beck and S. Robins, ``Computing the continuous discretely'', Undergraduate Texts in Mathematics. Springer, New York, second edition, 2015. 
%\bibitem{BG2} W. Bruns and J. Gubeladze, Polytopes, rings and K-theory, Springer Monographs in Mathematics. Springer, Dordrecht, (2009).  
\bibitem{BH} W. Bruns and J. Herzog, ``Cohen-Macaulay rings, revised edition'', Cambridge University Press, 1998. 
\bibitem{DH} E. De Negri and T. Hibi, Gorenstein algebras of Veronese type, {\em J. Algebra} {\bf 193} (1997), 629--639. 
\bibitem{Diestel} R. Diestel, ``Graph Theory'', Fifth edition, Graduate Texts in Mathematics, {\bf 173}. Springer, Berlin, 2017. 
\bibitem{GTT} S. Goto, R. Takahashi, N. Taniguchi, Almost Gorenstein rings -- towards a theory of higher dimension, {\em J. Pure Appl. Algebra} {\bf 219}, (2015), 2666--2712.
\bibitem{M2} D. Grayson and M. Stillman. Macaulay2, a software system for research in algebraic geometry, Available at {\tt http://www.math.uiuc.edu/Macaulay2/}.
\bibitem{HHO} J. Herzog, T. Hibi and H. Ohsugi, Binomial ideals, Graduate Texts in Mathematics, {\bf 279}. Springer, Cham, (2018). 
\bibitem{H87} T. Hibi, Distributive lattices, affine semigroup rings and algebras with straightening laws. In: Nagata, M., Matsumura, H. (eds.) Commutative Algebra and Combinatorics. Advanced Studies in Pure Mathematics, vol. 11, pp. 93--109. North-Holland, Amsterdam (1987). 
%\bibitem{HL16} T. Hibi and N. Li, Unimodular Equivalence of Order and Chain polytopes, {\em Math. Scand.} {\bf 118}, No. 1 (2016), 5--12. 
\bibitem{H} A. Higashitani, Almost Gorenstein homogeneous rings and their $h$-vectors, {\em J. Algebra} {\bf 456} (2016), 190--206. 
\bibitem{HM} A. Higashitani and K. Matsushita, Conic divisorial ideals and non-commutative crepant resolutions of edge rings of complete multipartite graphs, arXiv:2011.07714. 
\bibitem{HN} A. Higashitani and Y. Nakajima, Conic divisorial ideals of Hibi rings and their applications to non-commutative crepant resolutions, {\em Selecta Math.} {\bf 25} (2019), 25pp. 
%\bibitem{IU} A. Ishii and K. Ueda, Dimer models and the special McKay correspondence, {\em Geom.Topol.} {\bf 19}, (2015), 3405--3466. 
%\bibitem{IW} O. Iyama and M. Wemyss, Maximal modifications and Auslander--Reiten duality for non-isolated singularities, {\em Invent. Math.} {\bf 197}(3), (2014), 521--586. 
\bibitem{M17} M. Miyazaki, On the generators of the canonical module of a Hibi ring: A criterion of level property and the degrees of generators, {\em J. Algebra} {\bf 480} (2017), 215--236. 
\bibitem{M18} M. Miyazaki, Almost Gorenstein Hibi rings, {\em J. Algebra} {\bf 493} (2018), 135--149. 
\bibitem{OH98} H. Ohsugi and T. Hibi, Normal polytopes arising from finite graphs, {\em J. Algebra} {\bf 207} (1998), 409--426. 
\bibitem{OH00} H. Ohsugi and T. Hibi, Compressed polytopes, initial ideals and complete multipartite graphs, {\em Illinois J. Math.} {\bf 44}, No. 2 (2000), 391--406. 
\bibitem{SVV} A. Simis, W. V. Vasconcelos and R. H. Villarreal, The integral closure of subrings associated to graphs, {\em J. Algebra} {\bf 199} (1998), 281--289. 
\bibitem{S77} R. P. Stanley, Cohen--Macaulay complexes, in: M. Aigner (Ed.), Higher Combinatorics, Reidel, Dor-drecht and Boston, (1977), 51--62. 
%\bibitem{S86} R. P. Stanley, Two Poset Polytopes, {\em Discrete Comput. Geom.} {\bf 1} (1986), 9--23. 
\bibitem{Villa} R. H. Villarreal, ``Monomial algebras'', Monographs and Research Notes in Mathematics. CRC Press, Boca Raton, FL, 2015.
\bibitem{Y} K. Yanagawa, Castelnuovo's Lemma and $h$-vectors of Cohen--Macaulay homogeneous domains, {\em J. Pure Appl. Algebra} {\bf 105}, (1995) 107--116. 
\end{thebibliography}
\end{document}